\documentclass[11pt, reqno, english]{amsart}  
\usepackage[utf8]{inputenc}
\usepackage[T1]{fontenc}
\usepackage{amsmath,amsthm}
\usepackage{amsfonts,amssymb}
\usepackage{url}
\usepackage{mathtools}  
\usepackage[colorlinks=true,urlcolor=blue,linkcolor=red,citecolor=magenta]{hyperref}
\usepackage{enumerate, paralist}
\usepackage{tikz-cd}
\usepackage[left=1.3in,right=1.3in,top=1in,bottom=1in]{geometry}

\usepackage{xstring,ifthen,tabularx}
\usepackage{tikz}
\newcommand*\rc[1]{\tikz[baseline=(char.base)]{ \StrLen{#1}[\mystr]\ifthenelse{ \mystr > 1}{\hspace{-.07cm}\node[shape=circle,color=red,text=black,draw,inner sep=0pt](char){#1};}{\hspace{-.06cm}\node[shape=circle,color=red,text=black,draw,inner sep=1pt](char){#1};}}}

\newcounter{commentcounter}

\theoremstyle{plain}
\newtheorem{theorem}{Theorem}
\newtheorem{lemma}[theorem]{Lemma}
\newtheorem{corollary}[theorem]{Corollary}

\theoremstyle{definition}
\newtheorem{definition}[theorem]{Definition}

\newcommand{\R}{\mathbb{R}}
\newcommand{\C}{\mathbb{C}}
\newcommand{\Z}{\mathbb{Z}}

\begin{document}

\title[Trapezoids and regular maps]{On inscribed trapezoids and affinely $3$-regular maps}



\author{Florian Frick}
\address[FF]{Dept.\ Math.\ Sciences, Carnegie Mellon University, Pittsburgh, PA 15213, USA \newline \indent Inst. Math., Freie Universit\"at Berlin, Arnimallee 2, 14195 Berlin, Germany}
\email{frick@cmu.edu} 

\author{Michael Harrison}
\address[MH]{Institute for Advanced Study, 1 Einstein Drive, Princeton, NJ  08540}
\email{mah5044@gmail.com} 

\thanks{FF was supported by NSF grant DMS-1855591, NSF CAREER grant DMS 2042428, and a Sloan Research Fellowship.  MH was supported by Mathematisches Forschungsinstitut Oberwolfach with an Oberwolfach Leibniz Fellowship and by the Institute for Advanced Study through the NSF Grant DMS-1926686.}


\begin{abstract}
\small
We show that any embedding $\R^d \to \R^{2d+2^{\gamma(d)}-1}$ inscribes a trapezoid or maps three points to a line, where $2^{\gamma(d)}$ is the smallest power of $2$ satisfying $2^{\gamma(d)} \geq \rho(d)$, and $\rho(d)$ denotes the Hurwitz--Radon function.  The proof is elementary and includes a novel application of nonsingular bilinear maps.  As an application, we recover recent results on the nonexistence of affinely $3$-regular maps, for infinitely many dimensions~$d$, without resorting to sophisticated algebraic techniques.  
\end{abstract}

\date{\today}
\maketitle

\section{Introduction}

The main results of this paper are as follows.

\begin{theorem}
\label{thm:main}
	Let $d\ge 1$ be an integer and let $n \le 2d + \rho(d)-1$. Then any embedding $f\colon \R^d \to \R^n$ inscribes a trapezoid or maps three distinct points to a line. 
\end{theorem}

Here $\rho(d)$ denotes the Hurwitz--Radon function and is defined as follows: decompose $d$ as the product of an odd number and $2^{4a+b}$ for $0 \leq b \leq 3$, then $\rho(d) \coloneqq 2^b + 8a$.   In particular, Theorem~\ref{thm:main} implies: 

\begin{corollary}
\label{cor:main}
	Any embedding $\R^2 \to \R^5$, $\R^4 \to \R^{11}$, or $\R^8 \to \R^{23}$ inscribes a trapezoid or maps three distinct points to a line.
\end{corollary}

With a small technical modification we obtain the following improved bound.

\begin{theorem}
\label{thm:main2}
	Let $d\ge 1$ be an integer and let $n \le 2d + 2^{\gamma(d)}-1$, where $2^{\gamma(d)}$ is the smallest power of $2$ satisfying $2^{\gamma(d)} \geq \rho(d)$.  Then any embedding $f\colon \R^d \to \R^n$ inscribes a trapezoid or maps three distinct points to a line. 
\end{theorem}

\begin{corollary}
\label{cor:main2}
	Any embedding $\R^{16}\to\R^{47}$ inscribes a trapezoid or maps three distinct points to a line.
\end{corollary}

The idea of the proofs is simple: given an embedding $f \colon \R^d \to \R^n$, we define a suitable test function which takes the value zero if and only if $f$ inscribes a trapezoid or maps three points to a line, and then we apply a Borsuk-Ulam type result to show that the test function must hit zero in the stated dimensions.   While this general idea is ubiquitous in topology,  we emphasize that our specific test function utilizes the Hurwitz--Radon function in a novel way to capture the geometry of the situation more adequately than the ``obvious'' test function; see Section~\ref{sec:testmap} for a detailed discussion.   Additional context and known results related to Theorems~\ref{thm:main} and~\ref{thm:main2} may be found in Section~\ref{sec:discussion}.

The test function makes use of the following equivalent definition of the Hurwitz--Radon function: $\rho(d)$ is the largest integer such that there exists a nonsingular bilinear map $B \colon \R^{\rho(d)} \times \R^d \to \R^d$; here nonsingularity of $B$ means that $B(x,y) = 0$ if and only if $x = 0$ or $y = 0$.   Nonsingular bilinear maps generalize the multiplication in the classical division algebras $\R, \C, \mathbb{H}, \mathbb{O}$ and were primarily studied in a series of articles by K.Y.\ Lam (e.g.\ \cite{Lam1,Lam6,Lam4,Lam3,Lam2,Lam5}) and by Berger and Friedland~\cite{BergerFriedland}.  They have appeared prominently in topology alongside the Hurwitz--Radon function; for example, a famous result of Adams states that there exist $q-1$ linearly independent tangent vector fields on the sphere $S^{d-1}$ if and only if $q \leq \rho(d)$ \cite{Adams}.  Nonsingular bilinear maps can be used to construct immersions of projective spaces (see e.g. \cite{James}) and have recently appeared in studies of skew fibrations (\cite{Harrison,Harrison3,OvsienkoTabachnikov,OvsienkoTabachnikov2}),  totally nonparallel immersions \cite{Harrison5}, and coupled embeddability \cite{FrickHarrison2}.   Nevertheless, the idea to use nonsingular bilinear maps to improve the effectiveness of a test function appears to be new. 

\section{The test function and the proofs of the main results}

Fix an integer $d \geq 1$, write $\rho = \rho(d)$,  let $B\colon \R^\rho \times \R^d \to \R^d$ be a nonsingular bilinear map, and let $X = (0,1) \times F_2(\R^d)$, where $F_2(M) \coloneqq \left\{(x,y) \in M \times M \mid x \neq y\right\}$ is the configuration space consisting of pairs of distinct points of a space $M$.   Now given an embedding $f \colon \R^d \to \R^n$, we define $\Phi\colon F_2(X) \times S^{\rho-1} \to \R^n$ by
    \begin{align*}
     \Phi((t_1,x_1,y_1),(t_2,x_2,y_2),z)
     = \ \ &  t_1[f(x_1+y_1+B(z,x_1-y_1))  - f(x_1+y_1-B(z,x_1-y_1))] \\
     - & t_2[f(x_2+y_2+B(z,x_2-y_2)) - f(x_2+y_2-B(z,x_2-y_2))],
    \end{align*}
and we show that this test function detects degeneracy in the following sense:

\begin{lemma}
\label{lem:zerotrap}
If zero is in the image of $\Phi$, then the embedding $f$ inscribes a trapezoid or maps three distinct points to a line.
\end{lemma}

\begin{proof}
If $\Phi((t_1,x_1,y_1),(t_2,x_2,y_2),z) = 0$, the four image points 
\begin{align*}
u_i = f(x_i+y_i+B(z,x_i-y_i)) \ \mbox{ and } \ v_i = f(x_i+y_i-B(z,x_i-y_i)), \ i = 1,2,
\end{align*}
satisfy the equation $t_1(u_1-v_1) = t_2(u_2-v_2)$. Equivalently,
    \begin{align}
     \label{eq:trap}
     \frac{t_1}{t_1+t_2}u_1+\frac{t_2}{t_1+t_2}v_2 = \frac{t_1}{t_1+t_2}v_1+\frac{t_2}{t_1+t_2}u_2.
    \end{align}
    That is, the diagonal from $u_1$ to $v_2$ and the diagonal from $v_1$ to $u_2$ intersect, and the point of intersection splits both diagonals in a $t_1$ to $t_2$ ratio. So $u_1, u_2, v_1, v_2$ are image points of $f$ which form a (possibly degenerate) trapezoid.  We check that at least three of these four points are distinct.
    
    To verify that $u_i \neq v_i$, note that $x_i \neq y_i$ by assumption, and so $B(z,x_i-y_i) \neq 0$ by nonsingularity, and then apply injectivity of $f$.
    
    Now assume for contradiction that $u_1 = u_2$ and $v_1 = v_2$.  Since $f$ is injective, we obtain
    \begin{align*}
        x_1+y_1+B(z,x_1-y_1) &= x_2+y_2+B(z,x_2-y_2), \ \text{and} \\
        x_1+y_1-B(z,x_1-y_1) &= x_2+y_2-B(z,x_2-y_2).
    \end{align*}
    Adding and substracting these equations, respectively, yields
    \[
    x_1 + y_1 = x_2 + y_2 \quad \text{and} \quad B(z,x_1-y_1) = B(z,x_2-y_2).
    \]
    By bilinearity and nonsingularity of~$B$, the second equation gives $x_1-y_1 = x_2-y_2$.  Together
    with the first equation, this yields $x_1 = x_2$ and $y_1 = y_2$.  Since $(t_1,x_1, y_1) \neq (t_2,x_2, y_2)$ by assumption, we have $t_1 \neq t_2$, which together with the assumptions $u_1 = u_2$ and $v_1=v_2$ and Equation~(\ref{eq:trap}) implies that $u_1=v_1$, a contradiction.  Therefore, $u_1 \neq u_2$ or $v_1 \neq v_2$.
    
    Similarly, if $u_1 = v_2$ and $u_2 = v_1$, then together with the assumption $t_1(u_1-v_1) = t_2(u_2-v_2)$, we obtain $t_1(u_1-v_1)=-t_2(u_1-v_1)$, impossible since $u_1 \neq v_1$ and $t_i > 0$ by assumption.  Thus, $u_1 \neq v_2$ or $u_2 \neq v_1$.
    
    Therefore Equation~(\ref{eq:trap}) is a non-trivial affine combination that involves at least three pairwise distinct points. In the degenerate situation, where $\{u_1, u_2, v_1, v_2\}$ is a set of three points, Equation~(\ref{eq:trap}) implies that these points lie on a common line.
\end{proof}

In Section~\ref{sec:testmap} we explain the geometric motivation which led to the definition of $\Phi$.

Now to prove Theorem~\ref{thm:main}, we need only show that $\Phi$ must hit zero whenever $n \leq 2d+\rho-1$.   We make use of a classical result in equivariant topology.   Consider the $(\Z/2)^2$-action on $S^m \times S^q$ defined by letting the first copy of $\Z/2$ act antipodally on~$S^m$ and the second copy of $\Z/2$ act antipodally on~$S^q$.   Similarly we define a $(\Z/2)^2$-action on $\R^{m+q}$ by letting both generators act by $x \mapsto -x$.   A map $S^m \times S^q \to \R^{m+q}$ that commutes with these actions is a \emph{$(\Z/2)^2$-map}. 

\begin{lemma}[Hopf~\cite{hopf1940}]
\label{lem:zeros}
    Let $m\ge 1$ and $q\ge 1$ be integers that do not share a one in any digit of their binary expansions. Then any $(\Z/2)^2$-map $S^m \times S^q \to \R^{m+q}$ has a zero.
\end{lemma}

Hopf's proof was one of the earliest applications of cohomology theory.  In modern language, a $(\Z/2)^2$-map which avoids zero induces a map in cohomology $H^*(\R P^{m+q-1};\Z/2) \to H^*(\R P ^m \times \R P^q;\Z/2)$ which sends the generator $\gamma_{m+q-1}$ to the sum of generators $\gamma_m + \gamma_q$.  Therefore $(\gamma_m + \gamma_q)^{m+q} = 0$, hence the binomial coefficient $\scriptsize \Big( \begin{array}{c} m+q \\ m \end{array} \Big)$ is even.  By Lucas' theorem, this occurs if and only if $m+q$ has a zero in a digit of its binary expansion in which $m$ has a one,  which occurs if and only if $m$ and $q$ share a one in some digit.

Now consider the $\Z/2$-action on $F_2(X)$ which swaps the two points of $X$.  Lemma~\ref{lem:zeros} has the following simple consequence.

\begin{corollary} \label{cor:hopf} Let $m\ge 1$ and $q\ge 1$ be integers that do not share a one in any digit of their binary expansions.  If there exists a $\Z/2$-equivariant embedding $S^m \to F_2(X)$, then (by restriction) any $(\Z/2)^2$-equivariant map $F_2(X) \times S^q \to \R^{m+q}$ has a zero.
\end{corollary}

\begin{proof}[Proof of Theorem~\ref{thm:main}]
Let $n = 2d + \rho -1$ and let $f\colon \R^d \to \R^n$ be an embedding.  By Lemma~\ref{lem:zerotrap}, we only need to show that $\Phi$ has a zero.   There exists an embedding $h \colon S^{2d} \to X$, since $X$ is a $(2d+1)$-dimensional manifold, and this induces a $\Z/2$-equivariant embedding $S^{2d}~\to~F_2(X) \colon z \mapsto (h(z),h(-z))$.  Moreover, the map $\Phi$ is $(\Z/2)^2$-equivariant, so by Corollary~\ref{cor:hopf}, it suffices to check that $2d$ and $\rho-1$ do not share any common ones in their binary expansions.

To verify this, write $d$ as the product of an odd number with $2^\ell$,  where $\ell = 4a + b$ and $b \in \{0,1,2,3\}$.   Then $\rho = 2^b + 8a = 2^{\ell-4a}+ 8a < 2^{\ell+1}$. The last $\ell+1$ digits in the binary expansion of $2d$ (corresponding to $2^\ell, 2^{\ell-1}, \dots, 2^0$) are zero, and only those digits may be non-zero for~$\rho-1$.
\end{proof}

Next we present the proof of Theorem~\ref{thm:main2}, which relies on a few small modifications to the previous proof.

\begin{proof}[Proof of Theorem~\ref{thm:main2}]
Let $C \colon \R^{2^{\gamma(d)}-\rho+1} \times \R^\rho \times \R^d \to \R^d$ be the trilinear map defined by $C(w,z,x) = B(w,B(z,x))$.   This is well-defined by restriction in the first factor, since $2^{\gamma(d)}-\rho+1 \leq \rho$ by assumption.  Moreover,  nonsingularity of $B$ yields nonsingularity of $C$, where by nonsingularity of $C$ we mean that $C(w,z,x) = 0$ if and only if one of the factors is zero. 

Using $C$ in place of $B$, we define the $(\Z/2)^3$-map $\Phi \colon F_2(X) \times S^{2^{\gamma(d)-\rho}} \times S^{\rho-1} \to \R^n$ by
 \begin{align*}
     \Phi((t_1,x_1,y_1),&(t_2,x_2,y_2),w,z) \\
     = \ \ & t_1[f(x_1+y_1+C(w,z,x_1-y_1))  - f(x_1+y_1-C(w,z,x_1-y_1))] \\
     - & t_2[f(x_2+y_2+C(w,z,x_2-y_2)) - f(x_2+y_2-C(w,z,x_2-y_2))].
    \end{align*}

The proof of Lemma~\ref{lem:zerotrap} is otherwise unchanged, and so we only need to check that $\Phi$ has a zero when $n \leq 2d + 2^{\gamma(d)} - 1$.

For this, we use the obvious generalization of Hopf's lemma (with nearly-identical proof): any $(\Z/2)^3$-map $S^{m_1} \times S^{m_2} \times S^{m_3} \to \R^{m_1+m_2+m_3}$ has a zero provided that no two of the $m_i$ share a one in any digit of their binary expansions.   This applies to the integers $m_1 = 2^{\gamma(d)}-\rho$, $m_2 = \rho - 1$, and $m_3 = 2d$.  Indeed, the integers $m_1$ and $m_2$ share no ones since they sum to $2^{\gamma(d)} - 1$, and the argument that $m_1$ and $m_3$ share no ones is identical to that for $m_2$ and $m_3$, which was given in the proof of Theorem~\ref{thm:main}.
\end{proof}

\section{Context and history: $k$-regular embeddings}
\label{sec:discussion}

The main theorem and corollary are best understood in the context of \emph{regular maps}, first defined and studied by Borsuk in 1957 \cite{Borsuk}.

\begin{definition} A continuous map $f\colon \R^d \to \R^n$ is called \emph{$k$-regular} if for any $k$ pairwise disjoint points $x_1, \dots, x_k \in \R^d$ the points $f(x_1), \dots, f(x_k)$ are linearly independent.
\end{definition}

We offer simple examples for small $d$ or $k$:
\begin{compactenum}
\item $d = 1$: the moment curve $\R \to \R^k \colon t \mapsto (1,t,t^2,\dots,t^{k-1})$ is $k$-regular, due to the nonvanishing of the Vandermonde determinant on the configuration space $F_k(\R)$;
\item $k=2$: the map $\R^d \to \R^{d+1} = \R^d \times \R \colon x \mapsto (x,1)$ is $2$-regular;
\item $k=3$: if $h \colon \R^d \to S^d$ is an embedding,  then $\R^d \to \R^{d+2} \colon x \mapsto (h(x),1)$ is $3$-regular.
\end{compactenum}

Aware of these basic examples, Borsuk posed the question: given $d \geq 1$ and $k \geq 2$, what is the smallest dimension $n = n(d,k)$ such that $\R^d$ admits a $k$-regular map to $\R^n$?  
It is not difficult to check that the target dimensions are optimal for the given values of $d$ or $k$ in the above examples, so that $n(1,k) = k$,  $n(d,2) = d+1$, and $n(d,3) = d+2$.

In addition to the obvious geometric appeal, the question historically attracted interest due to connections with approximation theory and Chebyshev polynomials.  For $k \geq 4$ the question has been studied by a number of mathematicians and has proven to be notoriously difficult.

The first nontrivial result, that $n(d,2k) \geq (d+1)k$, appeared in a 1960 paper of Boltyansky, Ryzhkov, and Shashkin \cite{boltyansky1960}.  This can be shown by considering, for $f \colon \R^d \to \R^n$,  the test function
\begin{align}
\label{eqn:phi}
\varphi \colon D_1 \times \cdots \times D_k \times (\R - \left\{0\right\})^k \to \R^n \colon (x_1,\dots,x_k,\lambda_1,\dots,\lambda_k) \mapsto \sum \lambda_i f(x_i),
\end{align}
where the $D_i$ are disjoint disks in $\R^d$.  The bound follows from the observation that if $f$ is $2k$-regular, $\varphi$ embeds its $((d+1)k)$-dimensional domain into $\R^n$. 

In 1978, Cohen and Handel observed in \cite{CohenHandel} that a $k$-regular map $f \colon \R^d \to \R^n$ induces an $S_k$-equivariant map
\[
F_k(\R^d) \to V_k(\R^n) \colon (x_1,\dots,x_k) \mapsto (f(x_1),\dots,f(x_k));
\]
here $V_k(\R^n)$ is the Stiefel manifold of $k$-tuples of linearly independent vectors in $\R^n$, and the symmetric group $S_k$ acts on each space by permuting elements of the $k$-tuple.  This observation highlighted the equivariant nature of the problem, and in some sense, the subsequent results for $k$-regular maps can be viewed as a yardstick by which to measure the advances in equivariant cohomology theory over the following decades.

For example,  in the same paper, Cohen and Handel showed that $n(2,k) \geq 2k-\alpha(k)$, where $\alpha(k)$ denotes the number of ones in the dyadic presentation of $k$.  Chisholm \cite{Chisholm} generalized this by showing that $n(2^\ell,k) \geq 2^\ell(k-\alpha(k))+\alpha(k)$.   Other results on $k$-regular maps, including some for other manifolds or simplicial complexes, were contributed by Bogatyi \cite{Bogatyi}, Handel \cite{Handel,Handel2,Handel3}, and Handel and Segal \cite{HandelSegal}.  The first strong existence results were obtained in 2019 using methods of algebraic geometry \cite{buczynski2019}. 

Finally in 2021,  strong obstructions were computed by Blagojevi{\'c}, Cohen, Crabb, L{\"u}ck, and Ziegler, as a counterpart to their massive breakthrough in understanding the mod-$2$ equivariant cohomology of configuration spaces:

\begin{theorem}[\cite{blagojevic2020}, Theorem 6.16]
\label{thm:bla}
Let $d \geq 2$, $k \geq 1$, and write $d = 2^t + e$ for some $t \geq 1$ and $0 \leq e \leq 2^t-1$.  Let $\epsilon(k)$ denote the remainder of $k$ mod $2$.  Then $n(d,k) \geq (d-e-1)(k-\alpha(k))+e(\alpha(k)-\epsilon(k))+k$.
\end{theorem}

The proof of this theorem relies on a lengthy technical argument which corrects a proof of a theorem of Hung; the details surrounding the error, as well as a history of relevant configuration space computations, are well-chronicled in the introduction of \cite{blagojevic2020}.

We are now equipped to discuss the relationship between regular maps and Theorem~\ref{thm:main}.  It is convenient to introduce some intermediate language: a continuous map $f \colon \R^d \to \R^n$ is called \emph{affinely} $(k-1)$-\emph{regular} if for every $k$ distinct points of $\R^d$, the images do not lie in any affine $(k-2)$-plane.   The relationship with $k$-regularity is simple:

\begin{lemma}
\label{lem:aff-reg}
    There is a $k$-regular map $\R^d \to \R^n$ if and only if there is an affinely $(k-1)$-regular map $\R^d \to \R^{n-1}$.
\end{lemma}

\begin{proof}
	If $f\colon \R^d \to \R^{n-1}$ is affinely $(k-1)$-regular, then $\R^d \to \R^n, x \mapsto (f(x),1)$ is $k$-regular. Conversely, given a $k$-regular map $f\colon \R^d \to \R^n$, we may arrange (by restricting the domain if necessary) that $f$ misses some affine hyperplane $H$.   We may assume that $H$ is given by~${x_n =0}$. Identify $\R^{n-1}$ with the affine hyperplane~${x_n=1}$, and define $g\colon \R^d \to \R^{n-1}$ by letting $g(x)$ be the intersection of the line spanned by $f(x)$ with the hyperplane~${x_n=1}$. Then $g$ is affinely $(k-1)$-regular.
\end{proof}

Therefore Theorem~\ref{thm:main2}, which prohibits affinely $3$-regular maps $\R^d \to \R^n$ when $n \leq 2d + 2^{\gamma(d)} - 1$, can be viewed as an obstruction to the existence of $4$-regular maps.

\begin{corollary} $n(d,4) \geq 2d + 2^{\gamma(d)} + 1$.
\end{corollary}

We compare this bound to those described above.   First, note that the bound is never worse than that obtained by using the test function $\varphi$.   Chisholm's bound, which applies only when $d$ is a power of $2$, gives $n(d,4) \geq 3d + 1$, which is better than our bound for $d \geq 32$.  Similarly, the bound of Theorem~\ref{thm:bla} is frequently better than our bound, although there are infinitely many values of $d$ for which the bounds agree.  For example,  when $d = 2^\ell - 2 = 2(2^{\ell-1}-1)$, $\rho(d) = 2 = 2^{\gamma(d)}$,  our bound yields $n(2^\ell - 2,4) \geq 2^{\ell+1} - 1$, which matches the bound obtained by Theorem~\ref{thm:bla}.  In this sense Theorem~\ref{thm:main2},  which not only prohibits $4$-regular embeddings, but imposes further geometric constraints on $4$ points, strengthens the results of Theorem~\ref{thm:bla} for infinitely many dimensions $d$.

\section{The development and geometry of the test function}
\label{sec:testmap}

Here we explain the steps leading to the construction of the test function $\Phi$, and we describe the sense in which $\Phi$ captures the geometry more adequately than the test function $\varphi$ defined in Equation (\ref{eqn:phi}).

Consider two disjoint disks $D_1$, $D_2$ in $\R^d$.   An affinely $3$-regular map $f$ induces an embedding of the space $A$ of non-zero affine combinations of $D_1$ and $D_2$ into $\R^n$, which gives a lower bound for $n$ for dimension reasons; this is the geometric idea captured by the test function $\varphi$. 

Now observe that we can improve this by moving the disks around in a special way to generate whole families of embeddings.  In particular, using a nonsingular bilinear map, we can find an $S^{\rho-1}$-family of maps of $X$ into $\R^n$:
\[
\Psi_z \colon X \to \R^n \colon (t,x,y) \mapsto  t[f(x+y+B(z,x-y))-f(x+y-B(z,x-y))],  \ \ z \in S^{\rho-1}
\]
Note that $\Psi_{-z} = -\Psi_{z}$, and $\Psi_z$ fails to be an embedding if and only if $\Phi(\cdot,z)$ hits zero.  Thus the proof of Theorem~\ref{thm:main} relies not on obstructing a single embedding, but obstructing the existence of $\Z/2$-equivariant families of embeddings; this is the geometric idea captured by the test function $\Phi$.  Similarly, the technical modification of $\Phi$ used in the proof of Theorem~\ref{thm:main2} relies on obstructing $(\Z/2)^2$-equivariant families of embeddings.

The idea to try to obstruct equivariant families of embeddings was inspired by the authors' recent development and study of the $\Z/2$-\emph{coindex of spaces of embeddings}; see~\cite{FrickHarrison}.   It is our belief that we have only scratched the surface of results of this form, and it would be interesting to see whether similar techniques could yield strong bounds for other values of $d$ and $k$.  We emphasize that the tools used in our proofs of Theorems~\ref{thm:main} and~\ref{thm:main2} predated much of the study of $k$-regular embeddings, especially the recent results which rely on technical advances in equivariant topology.  We hope that our simple argument demonstrates that strong obstructions can be computed for $k$-regular embeddings and other nondegenerate functions without the use of sophisticated algebraic techniques; it is only important that the test functions adequately capture the geometry of the situation.   

We conclude by noting that Theorems~\ref{thm:main} and~\ref{thm:main2} also connect to the studies of inscribed shapes in various objects, perhaps most notably the square/rectangle peg problem in $\R^2$ (see \cite{GreeneLobb} or \cite{Frick} for description and history).  According to Greene and Lobb in \cite{GreeneLobb},  there is ``a long line of attack on these problems which involves identifying the inscribed feature with the (self-)intersection of an associated geometric-topological object. The arguments tend to be quite short, once the appropriate outlook and auxiliary result is identified.''  Our short proof of Theorem~\ref{thm:main2}, which capitalizes on the fact that a self-intersection occurs in one of a $(\Z/2)^2$-equivariant ($S^{2^{\gamma(d)} - \rho(d)} \times S^{\rho(d)-1}$)-parameter family of maps $X \to \R^n$ when $n \leq 2d + 2^{\gamma(d)} - 1$, provides one more example of such an attack.

\bibliographystyle{plain}

\end{document}